\documentclass[12pt]{article}
\usepackage{amssymb,amsmath, amsthm, amscd,ifthen}
\usepackage[T2A]{fontenc}
\usepackage[cp1251]{inputenc}
\usepackage[russian,english]{babel}
\usepackage[dvips]{graphicx}
\usepackage{indentfirst}

\newcommand{\N}{{\mathbb N}}

\newtheorem{prob}{Problem}

\newtheorem{cor}{Corollary}
\newtheorem{thm}{Theorem}

\newtheorem{lem}{Lemma}

\date{}
\title{On random subgraphs of Kneser and Schrijver graphs}
\begin{document}
\author{Andrey Kupavskii\footnote{Moscow Institute of Physics and Technology, \'Ecole Polytechnique F\'ed\'erale de Lausanne; Email: {\tt kupavskii@yandex.ru} \ \ Research supported in part by the Swiss National Science Foundation Grants 200021-137574 and 200020-14453 and by the grant N 15-01-03530 of the Russian Foundation for Basic Research.}}

\maketitle

\begin{abstract} A Kneser graph $KG_{n,k}$ is a graph whose vertices are in one-to-one correspondence with $k$-element subsets of $[n],$ with two vertices connected if and only if the corresponding sets do not intersect. A famous result due to Lov\'asz states that the chromatic number of a Kneser graph $KG_{n,k}$ is equal to $n-2k+2$. In this paper we study the chromatic number of a random subgraph of a Kneser graph $KG_{n,k}$ as $n$ grows. A random subgraph $KG_{n,k}(p)$ is obtained by including each edge of $KG_{n,k}$ with probability $p$. For a wide range of parameters $k = k(n), p = p(n)$ we show that $\chi(KG_{n,k}(p))$ is very close to $\chi(KG_{n,k}),$  w.h.p. differing by at most 4 in many cases. Moreover, we obtain the same bounds on the chromatic numbers for the so-called Schrijver graphs, which are known to be vertex-critical induced subgraphs of Kneser graphs. \end{abstract}

\section{Introduction}
In this paper we study the well-known Kneser graph $KG_{n,k}.$ Vertices of the graph are $k$-subsets of an $n$-element set, which we denote by $[n]$. Two $k$-sets are joined by an edge if they are disjoint. These graphs were first investigated by Martin Kneser \cite{Knez}. He showed that $\chi(KG_{n,k})\le n-2k+2$ and conjectured that this bound is tight. The conjecture was proved by L\'aszl\'o Lov\'asz \cite{Lova} over 20 years later. He used tools from algebraic topology, giving birth to the field of topological combinatorics.
Later, two very nice and short proofs were given by Imre B\' ar\' any \cite{Bar} and Joshua E. Greene \cite{Gr}.

Several papers were devoted to the study of the chromatic number of Kneser graphs of set systems. For any system of $k$-sets $\mathcal A\subset {[n]\choose k}$ we can define the Kneser graph $KG(\mathcal A)$ in the following natural way. The vertices of $KG(\mathcal S)$ are the elements of $\mathcal A$, while two of them are joined if and only if they are disjoint. In particular, there were results by V. Dol'nikov \cite{Dol} and A. Schrijver \cite{Sch}. Later, some of the results were extended to Kneser hypergraphs (see \cite{AFL}).

In the paper \cite{Sch}, Schrijver noted that a slight modification of B\' ar\'any's proof allow to prove a much stronger statement: $\chi(SG_{n,k})=n-2k+2.$ Here $SG_{n,k}$ is the subgraph of $KG_{n,k}$ on the vertex set induced by all the \textit{stable }$k$-element subsets $S$ of $[n]$, that is, the ones that do not  contain two cyclically consecutive elements from $[n]$: if $i\in S,$ then $(i-1)\!\!\mod n, (i+1)\!\!\mod n$ are not in $S$. It looks really surprising, since the number of vertices in $SG_{n,k}$ is much smaller than in $KG_{n,k}$. Schrijver also proved that $SG_{n,k}$ is a vertex-critical subgraph of $KG_{n,k}$, which means that any proper induced subgraph of $SG_{n,k}$ has a strictly smaller chromatic number.

 A Kneser graph of any $k$-uniform set system is an induced subgraph of $KG_{n,k}$. We, in turn, deal with spanning subgraphs of Kneser and Schrijver graphs. Roughly speaking, we show that the chromatic number of random subgraphs of these graphs is very close to the chromatic number of the original graphs. Speaking more formally, we have to define the random graph model for the problem. We do it for Kneser graphs, with the definition for Schrijver graphs being analogous. A random graph $KG_{n,k}(p)$ has the same set of vertices as $KG_{n,k},$ and each edge from $KG_{n,k}$ is included in $KG_{n,k}(p)$ with probability $p$. Stability questions for random graphs got significant attention in the last years, with many beautiful results having been proved. We mention only the most relevant. In the paper \cite{BNR} Bollob\'as, Narayanan and Raigorodskii studied the size of maximal independent sets in $KG_{n,k}(p)$, showing a strong stability result for $k$ being not too large: $k = o(n^{1/3})$. Later on, this condition was relaxed to $k<(1/2-\epsilon)n$ by Balogh, Bollob\'as, and Narayanan \cite{BBN}.

 The case of more general graphs $K(n,k,l)$ with $0<l<k$ was studied by Bogolyubskiy, Gusev, Pyaderkin and Raigorodskii in \cite{BGPR,BGPR2}. The vertices of $K(n,k,l)$ are $k$-element subsets of $[n]$, with two vertices adjacent if the corresponding sets intersect in exactly $l$ elements. In \cite{BGPR,BGPR2} the authors obtained several results concerning the independence number and the chromatic number of random subgraphs of such graphs. For a bit broader perspective on extremal questions for random graphs we refer the reader to the survey \cite{RS}. Different questions concerning random graphs are discussed in the books \cite{AS}, \cite{Boll}.

\section{The main theorem and its corollaries}
 The main result of the paper is the following quite technical theorem. 
\begin{thm}\label{th1}  Let $k = k(n)$, $\ell=\ell(n)$ be integer functions. Assume that  $k\ge 2, \ell\ge 1$ and choose $p = p(n)$ such that $0<p\le 1$. Put $d = n-2k-2\ell+1$ and $t =\left\lceil {\ell+k\choose k}/d\right\rceil$. If $d\ge 2$ and for some $\epsilon>0$ we have $(1-\epsilon)p>t^{-2}n\ln 3 + 2t^{-1}(1 +\ln d),$ then the graph $SG_{n,k}(p)$ w.h.p. has chromatic number at least $d+1$.
\end{thm}

\textbf{Remark 1.} We formulate and prove all the results for $SG_{n,k}(p)$, since they yield the same bounds for $KG_{n,k}(p)$. Thus, in Theorem \ref{th1}, as well as in the forthcoming corollaries, $SG_{n,k}(p)$ may be replaced by $KG_{n,k}(p)$.\\

\textbf{Remark 2.} We use the following notations. By $\ln x$ we denote the natural logarithm of $x$. By writing $f(n)\gg g(n)$ we mean that $\lim_{n\to \infty}g(n)/f(n) = 0$, which is equivalent to writing $g(n)\ll f(n)$. A formula $f(n)\sim g(n)$ signifies that $\lim_{n\to \infty}f(n)/g(n) = 1$. We remind the reader that the abbreviation w.h.p. means ``with high probability'', that is, with probability tending to 1 as $n$ tends to infinity.\\

The condition of Theorem \ref{th1} is quite difficult to interpret, since there are too many parameters involved. Before proving the theorem we  present several corollaries of Theorem \ref{th1}, which follow directly from it by easy calculations.

\begin{cor}\label{cor1} Let $p$ be fixed, $0<p\le 1$. If $k\gg n^{3/4}$, then w.h.p. $\chi(SG_{n,k}(p))\ge \chi(KG_{n,k})-4$. Moreover, if $n-2k \ll \sqrt n,$ then w.h.p. $\chi(SG_{n,k}(p))\ge \chi(KG_{n,k})-2.$
\end{cor}
\begin{proof} We start with the first claim. It is enough to verify the condition on $p$ from Theorem \ref{th1} for sufficiently large $n$ and $\ell=2$. In these assumptions we have $t = \left\lceil {2+k\choose 2}/d\right\rceil\gg (n^{3/4})^2/n = n^{1/2}$. Thus, we have $t^{-2}n\ln 3 + 2t^{-1}(1 +\ln d) \le t^{-2}n\ln 3 + 2t^{-1}(1 +\ln n)\ll 1,$ so the condition on $p$ holds for sufficiently large $n$.

As for the second claim, it is enough to verify the condition on $p$ for sufficiently large $n$ and $\ell=1$. We have $d\ll \sqrt n$. We again obtain the same bound on $t$: $t = \left\lceil {k+1\choose 1}/d\right\rceil\gg n/\sqrt n = n^{1/2}$. The rest remains unchanged.
\end{proof}

\begin{cor}\label{cor2} Fix some function $g(n), g(n)\to\infty$ as $n\to \infty$. For any $k$, where $2\le k\le \frac{n}2 -g(n),$ and for any fixed $p$, $0<p\le 1$ w.h.p. we have $\chi(SG_{n,k}(p))\sim \chi(KG_{n,k})$.
\end{cor}

This is an intriguing corollary, since Kneser graph $KG_{n,k}$ for $k=1$ is just a complete graph $K_n$, and it is known that w.h.p. $\chi (K_n(1/2))  \sim \frac n{2\log_2 n}\ll n = \chi(K_n)$. It is clear that our proof cannot give any reasonable bound in the case when $k=1$. Indeed, to fulfill the condition on $p$ one has to make $t=\lceil (l+1)/d\rceil $ bigger than $\sqrt n$, for which, in turn, one has to choose $d\le \frac 12\sqrt n$. Therefore, the best we can obtain is the square root bound: $\chi(K_n(1/2))\ge c\sqrt n$ for some constant $c>0$.
\begin{proof} Corollary \ref{cor1} justifies the assertion of Corollary \ref{cor2} in the case $k\gg n^{3/4}$. Taking a sequence of $k(n)$, we may split it into at most two infinite sequences $N_1,N_2, N_1\cup N_2 = \N$, such that  $k(n)\gg n^{3/4}$ if we take a limit over $n\in N_1$, and $k(n)\ll n^{4/5}$ if we take a limit over $n\in N_2$. The first sequence we treat as the case of $k\gg n^{3/4}$. Therefore, w.l.o.g., we may assume that $N_2 = \N$ and, thus, $k\ll n^{4/5}.$ In this case we have $\chi(KG_{n,k})\sim n$. It is thus sufficient to verify the condition on $p$ with some $\ell$ that satisfies  $n^{3/4}\ll \ell \ll n$. As in the proof of Corollary \ref{cor1}, we obtain a similar bound on $t$: $t = \left\lceil {k+\ell\choose k}/d\right\rceil\ge \left\lceil {2+\ell\choose 2}/n\right\rceil \gg (n^{3/4})^2/n = n^{1/2}.$ The rest remains the same.
\end{proof}

The next two corollaries are a bit more general. In particular, we allow for $p$ to be a function of $n$. We omit the proof of Corollary \ref{cor4}, as it repeats the proof of Corollary \ref{cor3}.

\begin{cor}\label{cor3} Fix a natural $\ell$. \\ 1. If $1\ge p\gg \frac {n^3}{k^{2\ell}}+\frac{n(1+\ln n)}{k^{\ell}}$, then w.h.p. $\chi(SG_{n,k}(p))\ge \chi(KG_{n,k})-2\ell$.\\
2. If $p$ is fixed, $0<p\le 1$, and $k\gg n^{\frac 3{2\ell}}$, then w.h.p. $\chi(SG_{n,k}(p))\ge \chi(KG_{n,k})-2\ell$.\\
\end{cor}
\begin{proof} 
In the assumptions of the corollary we have $t = \left\lceil {l+k\choose l}/d\right\rceil\ge ck^l/n,$ where $c>0$ is a constant. Thus, we have $t^{-2}n\ln 3 + 2t^{-1}(1 +\ln d) \le O\Bigl(\frac{n^3}{k^{2l}} + \frac{n(1+\ln n)}{k^{\ell}}\Bigr) \ll p,$ so the condition on $p$ from Theorem \ref{th1} holds for sufficiently large $n$.

As for the second part, it is enough to note that if  $k\gg n^{\frac 3{2\ell}}$, then both $\frac{n^3}{k^{2l}} \ll 1$ and $\frac{n(1+\ln n)}{k^{\ell}}\ll 1$ and the condition from the first part is fulfilled.
\end{proof}

Due to the symmetric role that $k$ and $\ell$ play in the condition on $p$, Corollary \ref{cor4} is the same as Corollary \ref{cor3}, but with $k$ and $\ell$ interchanged. 

\begin{cor}\label{cor4} Fix a natural $k$. \\ 1. If $1\ge p\gg \frac {n^3}{\ell^{2k}}+\frac{n(1+\ln n)}{\ell^k}$, then w.h.p. $\chi(SG_{n,k}(p))\ge \chi(KG_{n,k})-2\ell$.\\
2. If $p$ is fixed, $0<p\le 1$, and $\ell\gg n^{\frac 3{2k}}$, then w.h.p. $\chi(SG_{n,k}(p))\ge \chi(KG_{n,k})-2\ell$.\\
\end{cor}

\section{Proof of Theorem 1}
 We rely on the proof of Kneser's conjecture due to I. B\' ar\'any (it is possible to obtain the result for Kneser graphs, but not for Schrijver graphs, out of Greene's proof), as it is presented in the book \cite{Mat}.
First, for the rest of the paper we fix a certain mapping $f: [n]\to S^{d-1},$ using the following theorem due to Gale (see \cite{Mat}):
\begin{lem} There exists a mapping $f:[n]\to S^{d-1}$, such that every open hemisphere in $S^{d-1}$ contains an image of a stable $(k+\ell)$-element subset of $[n]$.
\end{lem}

This is the point where we use the definition of $d$: $d = n-2k-2\ell+1$. We note that such a mapping can be relatively easily constructed using the moment curve. In particular, this mapping is such that no $d$ points lie on a great (hyper)sphere, that is, a sphere that is an intersection of $S^d$ and a hyperplane that passes through the center of $S^d$. In what follows we omit the prefix ``hyper''. Consider all combinatorially distinct partitions of the set $f([n])$ by great spheres. \\

The first part of the proof is dedicated to an (a bit technical) estimation of the probability for a random graph $SG_{n,k}(p),$ which we for shorthand denote by $G$, and a given function $f$, to have a certain partition.

Combinatorially, each partition by a great sphere is determined by the subsets $S^+, S^-$ of points from $[n],$ such that $f(S^+), f(S^-)$ lie into different open hemispheres (with respect to a given great sphere).
Fix a great sphere.
Consider the sets $\mathcal S^+_k, \mathcal S^-_k$ of all stable $k$-element subsets of $S^+,S^-$ respectively. Note that $|\mathcal S^+_k|, |\mathcal S^-_k|\ge {k+\ell\choose k}$, since $S^+, S^-$ contain a $(k+\ell)$-element stable subset, any of which $k$-subsets is stable as well. Put $t(S^+) = \lceil|\mathcal S^+_k|/d\rceil, t(S^-) = \lceil|\mathcal S^+_k|/d\rceil.$ We choose a $t(S^+)$-element subset $M^+$ and a $t(S^-)$-element subset $M^-$ out of $\mathcal S^+_k,\mathcal S^-_k$ respectively. There are ${|\mathcal S^+_k|\choose t(S^+)}{|\mathcal S^-_k|\choose t(S^-)}$ choices for a pair of $M^+,M^-$, which is at most ${t(S^+)d\choose t(S^+)}{t(S^-)d\choose t(S^-)}$.

Below is the description of an event $A$, that is crucial for the proof. The event $A$ holds if there exists a partition of the set $f([n])$ by a great sphere and subsets $M^+, M^-$ of the type described above, so that there is no edge between them in $G$. We can bound the probability of $A$ as follows:

\begin{equation}\label{eq1}
P(A)\le \sum_{\text{possible partitions into } S^+,\ S^-}{t(S^+)d\choose t(S^+)}{t(S^-)d\choose t^(S^-)}(1-p)^{t(S^+)t(S^-)}\end{equation}

Let us show that for $t_1=t_2=\lceil {k+l\choose k}/d\rceil$  the value of the function $$g(t_1,t_2)={t_1d\choose t_1}{t_2d\choose t_2}(1-p)^{t_1t_2}$$ is maximal among all $t_1,t_2\ge \lceil {k+\ell\choose k}/d\rceil$. Since $t_1,t_2$ are integer and the expression is symmetric with respect to $t_1,t_2$, it is enough to compare the values of the expression for $t_1,t_2$ and $t_1+1, t_2$:
\begin{align*}\frac {g(t_1+1,t_2)}{g(t_1,t_2)}=&\ \frac{((t_1+1)d)!}{(t_1+1)!((t_1+1)(d-1))!} \frac{t_1!(t_1(d-1))!}{(t_1d)!}(1-p)^{t_2} = \\=&\ \frac{(t_1+1)d\prod_{i=1}^{t_1}(((t_1+1)(d-1)+i)}{(t_1+1)\prod_{i=1}^{t_1}((t_1(d-1)+i)}(1-p)^{t_2}<\\<& \ d\left(\frac{t_1+1}{t_1}\right)^{t_1}(1-p)^{t_2}<  e^{1+\ln d-pt_2}.\end{align*}
The last expression is less than 1 if $p>\frac {(1+\ln d)}{\lceil {k+\ell\choose k}/d\rceil}$, which is a corollary of the condition on $p$ given in the theorem. Therefore, we obtain the following bound by changing in (\ref{eq1}) all the summands $g(t_1,t_2)$ to $g(t,t)$ (in the following calculation we use the notation $t = \lceil {k+\ell\choose k}/d\rceil$, which we used in the formulation of Theorem \ref{th1}):

\begin{align*}
P(A)<&\  3^n{td\choose t}^2(1-p)^{t^2}< 3^n\exp\left\{2t(1+\ln d) -pt^2\right\} =\\
=&\ \exp\left\{t^2\bigl(t^{-2}n\ln 3 + 2t^{-1}(1+\ln d) -p\bigr)\right\}<\\<&\ \exp\left\{-t^2\epsilon p\right\}<\exp\{-\epsilon n\ln 3\}\ll 1.
\end{align*}

The first inequality follows from the fact that the number of possible partitions is bounded from above by $3^n$.
In the second inequality we use the bound ${td\choose t}<\left(\frac{etd}{t}\right)^t$. In the last two inequalities we used the condition on $p$ from the statement of Theorem \ref{th1}.\\

 The second part of the proof interconnects the event $A$ and the property of the graph $G$ to have chromatic number at most $d$.

 To complete the proof of the theorem, we  show that in any (not necessary proper) coloring of $SG_{n,k}$ into $d$ colors we always can find two such sets of vertices $M^+, M^-$, which additionally are monochromatic. Since $P(A)\ll 1$, w.h.p. any two such sets in $G$ have at least one edge between them, and it is thus impossible to find a proper $d$-coloring of $G$.

Consider an arbitrary coloring of $SG_{n,k}$ into $d$ colors. Each coloring together with the mapping $f$ described in the beginning of the proof induce a partition of the sphere $S^{d-1}$ into sets $B_1,\ldots, B_d$. We describe the partition below.
 A point $x$ of $S^{d-1}$ is included into $B_i$, where $1\le i\le d$, if in the open hemisphere with the center in $x$  the color $i$  is the most popular color among all the stable $k$-sets fully contained in that hemisphere. If there are equally popular colors, we put the point in all such $B_i$'s. It means that there are at least $t$ stable $k$-sets colored in $i,$ lying in the open hemisphere with the center in $x$.

    We apply the celebrated theorem due to Lusternik, Schnirelman and Borsuk (see the book \cite{Mat}), one of which formulations is as follows:
\begin{thm}\label{bor}
     Whenever the sphere $S^{d-1}$ is covered by open sets $S_1,\ldots, S_{d}$, there exists $i$ such that $S_i\cap (-S_i)\neq \emptyset.$
\end{thm}

We claim that the sets $B_1,\ldots, B_d$ are open. Indeed, let us check that $B_1$ is open. For any given point $y \in f([n])$ the set $U_y$ of points  $x\in S^d,$ such that the open hemisphere with the center in  $x$ contains $y,$ is open. Next, for any subset $S\subset f([n])$ the set $U_S=\cap_{y\in S}U_y$ is also open. Finally, let $\mathcal S_1$ be the set of all subsets $S$ with the property that in the coloring of $k$-sets of $S$ the color 1 is the most popular. Then $B_1 = \cup_{S\in \mathcal S_1} U_S,$ and, consequently, $B_1$ is open.

Thus, we can apply Theorem \ref{bor} and obtain that one of $B_i$ contains two antipodal points $x, -x$. Then, by the definition of the set $B_i$, there are at least $t(S^+),t(S^-)$ vertices of color $i$ in the sets $\mathcal S^+_k,$ $\mathcal S^-_k$ respectively. Here the sets $\mathcal S^+_k,$ $\mathcal S^-_k$ correspond to the partition of the sphere by the great sphere, orthogonal to $\bar x.$ We put $M^+, M^-$  equal to these two monochromatic sets, and, by the first part of the proof, w.h.p. they have an edge between them, so the proof is complete.

\section{Discussion} The results on the chromatic number obtained in the paper are quite sharp, but it would be very interesting to make them even sharper. In particular, we formulate the following problem:

\begin{prob} Is it true that for some $k = k(n)$ w.h.p. we have $\chi(KG_{n,k}(1/2))= \chi(KG_{n,k})$?
\end{prob}

I expect that the answer should be positive for a wide range of $k$. The same question may be asked for Schrijver graphs, in which case it is not so complicated to see that the answer is negative. It almost surely drops by 1 for a wide range of $k$. However, a question if it drops by more than 1 is open. To approach these problems one may need to modify the original Lov\'asz' technique for this probabilistic setting. It is as well very interesting to figure out, how the chromatic number of random Kneser hypergraphs behaves.

\section{Acknowledgements} I am grateful to the anonymous referees for their valuable comments that helped to improve the presentation and the content of the paper. In particular, they encouraged me to think about this problem for Schrijver graphs, which led to the generalization of the result.

\end{document}